\documentclass[a4paper, 10pt, twoside, notitlepage]{amsart}

\usepackage{amsmath,amscd}
\usepackage{amssymb}
\usepackage{amsthm}
\usepackage{comment}
\usepackage{graphicx, xcolor}

\usepackage{mathrsfs}
\usepackage[ocgcolorlinks, linkcolor=blue]{hyperref}

\usepackage{bm}
\usepackage{bbm}
\usepackage{url}

\usepackage[utf8]{inputenc}
\usepackage{mathtools,amssymb}
\usepackage{esint}
\usepackage{tikz}
\usepackage{dsfont}
\usepackage{relsize}
\usepackage{url}
\usepackage{xcolor}
\usepackage{graphicx}
\usepackage{mathrsfs}
\usepackage[shortlabels]{enumitem}
\usepackage{lineno}
\usepackage{amsmath}
\usepackage{enumitem}
\usepackage{amsthm} 
\usepackage{verbatim}
\usepackage{dsfont}
\numberwithin{equation}{section}

\allowdisplaybreaks

\mathtoolsset{showonlyrefs}

\graphicspath{{images/}}

\newtheorem{theorem}{Theorem}[section]
\newtheorem{lemma}[theorem]{Lemma}
\newtheorem{definition}[theorem]{Definition}
\newtheorem{corollary}[theorem]{Corollary}

\newtheorem{proposition}[theorem]{Proposition}
\newtheorem{question}{Question}
\newtheorem{remark}[theorem]{Remark}

\title[An inverse problem for fractional $p\,$-Laplace equations]{
	Determining coefficients for a fractional $p$-Laplace equation from exterior measurements 
}

\author[M. Kar]{Manas Kar}
\address{Indian Institute of Science Education and Research (IISER) Bhopal, India}
\email{manas@iiserb.ac.in}

\author[Y.-H. Lin]{Yi-Hsuan Lin}
\address{Department of Applied Mathematics, National Yang Ming Chiao Tung University, Hsinchu, Taiwan}
\email{yihsuanlin3@gmail.com}

\author[P. Zimmermann]{Philipp Zimmermann}
\address{Department of Mathematics, ETH Zurich, Z\"urich, Switzerland}
\email{philipp.zimmermann@math.ethz.ch}

\newcommand{\R}{{\mathbb R}}

\newcommand{\N}{{\mathbb N}}

\newcommand{\eps}{\varepsilon}

\newcommand{\var}{\varphi}

\DeclareMathOperator{\DDiv}{div} 
\DeclareMathOperator{\Div}{Div} 
\DeclareMathOperator{\supp}{supp} 
\DeclareMathOperator{\dist}{dist} 
\newcommand{\oddvf}[1]{L^0\left(\bigwedge_{od}^1 #1 \right)} 

\begin{document}
	
	\maketitle
	\begin{abstract} 
		We consider an inverse problem of determining the coefficients of a fractional $p\,$-Laplace equation in the exterior domain. Assuming suitable local regularity of the coefficients in the exterior domain, we offer an explicit reconstruction formula in the region where the exterior measurements are performed. This formula is then used to establish a \emph{global uniqueness} result for real-analytic coefficents. In addition, we also derive a stability estimate for the unique determination of the coefficients in the exterior measurement set.
		
		\medskip
		
		\noindent{\bf Keywords.} Inverse problems, exterior determination, fractional gradient, fractional divergence, fractional $p\,$-Laplacian.
		
		\noindent{\bf Mathematics Subject Classification (2020)}: Primary 35R30; secondary 26A33, 42B37, 46F12

	\end{abstract}

	\tableofcontents
	
	\section{Introduction}
	
	In this article, we study an inverse problem for a fractional $p\,$-Laplace equation. To formulate the problem, let us consider a partial differential equation (PDE) of the form
	\begin{equation}
		\label{eq: fractional p laplace type equation}
		\Div_s(\sigma|d_s u|^{p-2}d_s u)=0,
	\end{equation}
	where $1<p<\infty$, $\sigma=\sigma(x,y)\colon\R^n\times\R^n\to\R$ satisfies the uniform ellipticity condition 
	\begin{equation}
		\label{eq: uniform ellipticity cond}
		\lambda \leq \sigma(x,y)\leq\lambda^{-1}, \text{ for all }x,y\in \R^n,
	\end{equation}
	and for some constant $\lambda>0$. Later we call this operator appearing in \eqref{eq: fractional p laplace type equation} \emph{weighted fractional $p\,$-Laplacian}. Here, $d_su$ denotes the \emph{fractional $s$-gradient} and $\Div_s$ is its adjoint, which is the \emph{fractional $s$-divergence}, with respect to the measure 
	$$
	d\mu=\frac{dxdy}{|x-y|^n} \text{ on } \R^{n} \times \R^{n}.
	$$
	More concretely, for any function $u\colon\R^n\to\R$ and $s\in (0,1)$, the fractional $s$-gradient and $s$-divergence could be defined by
	\begin{align*}
		d_su(x,y)=\frac{u(x)-u(y)}{|x-y|^s} \quad \text{ and }\quad 
		\left\langle \Div_s u,\varphi \right\rangle =\int_{\R^{2n}}\frac{u(x,y)d_s\varphi(x,y)}{|x-y|^n}\, dxdy,
	\end{align*}
	respectively, for $x,y\in\R^n$ and for all $\varphi\in C_c^{\infty}(\R^n)$ (see for example \cite{MazowieckaSchikorra2018} and Section \ref{sec: Preliminaries} for detailed definitions). The weighted fractional $p\,$-Laplacian is weakly defined by 
	\begin{align*}
		&\left\langle \Div_s(\sigma|d_su|^{p-2}d_su), \varphi \right\rangle\\
		:=&\int_{\R^{2n}}\sigma(x,y) |u(x)-u(y)|^{p-2}\frac{(u(x)-u(y))(\varphi(x)-\varphi(y))}{|x-y|^{n+sp}} \, dxdy,
	\end{align*}
	for any $\varphi \in C^\infty_c(\R^n)$.

	Next, we are going to formulate an inverse problem related to \eqref{eq: fractional p laplace type equation}. The observations of our inverse problem are encoded in the \emph{Dirichlet-to-Neumann} (DN) map, which is formally defined by
	\begin{equation}
		\label{eq: formal DN map}
		\Lambda_\sigma(f)=\left.\Div_s(\sigma|d_su|^{p-2}d_su)\right|_{ \Omega_e}
	\end{equation}
	where $\Omega_e:=\R^n\setminus\overline{\Omega}$ denotes the exterior domain and $u_f$ is the unique solution of
	\begin{equation}\label{eq: fractional-p-Lap equation}
		\begin{cases}
			\Div_s(\sigma|d_su|^{p-2}d_su)=0& \text{ in }\Omega,\\
			u=f\ &\text{ in } \Omega_e.
		\end{cases}
	\end{equation}
	Here we have assumed the well-posedness of \eqref{eq: fractional p laplace type equation} at the moment (the proof will be given in Section \ref{sec: The forward problem and DN map}) and that the weighted fractional $p\,$-Laplacian of $u$ induces at least a distribution on $\Omega_e$.
	Then we ask the following question:
	
	\begin{question}
		Let $W\subset \Omega_e$ be a given nonempty open set and assume that the coefficients $\sigma_1, \sigma_2$ satisfy $\left.\Lambda_{\sigma_1}f\right|_{W}=\left.\Lambda_{\sigma_2}f\right|_{W}$, for all $f\in C_c^{\infty}(W)$. If $\Sigma_j\colon\R^n\to\R$ is given by $\Sigma_j(x)=\sigma_j(x,x)$ for $j=1,2$, can we conclude $\Sigma_1=\Sigma_2$ in $W$?
	\end{question}
	
	In the special case of coefficients of the form $\sigma(x,y)=\gamma^{1/2}(x)\gamma^{1/2}(y)$ this generalizes recent results for the fractional conductivity equation (i.e. $p=2$) by the last author (see \cite{RZ2022unboundedFracCald,RZ2022FracCondCounter,CRZ2022global, RZ2022LowReg, StabilityFracCondEq} for the elliptic and \cite{LRZ22} for the parabolic case).  

	As $s=1$, the related inverse problem to the equation 
	\begin{equation}
		\label{eq: p Laplace type equations}
		\text{div}(\gamma|\nabla u|^{p-2}\nabla u)=0
	\end{equation} is the so called (classical) \emph{$p\,$-Calder\'on problem}, where $\gamma\colon\R^n\to\R$ is a positive scalar function\footnote{In this paper, we use the notation $\gamma=\gamma(x):\R^n \to \R$ and $\sigma=\sigma(x,y):\R^{2n}\to \R$.}, and we next discuss about it.
	\subsection{The \texorpdfstring{$p\,$}{p}-Calder\'on problem}
	
	In the $p\,$-Calder\'on problem the DN map is strongly given by 
	\[
	f\mapsto \Lambda_{\gamma}^p f=\left.\gamma|\nabla u_f|^{p-2}\partial_{\nu} u_f\right|_{\partial\Omega},
	\]
	where $u_f$ is the unique solution to
	\begin{equation}
		\label{eq: p Calderon problem summary}
		\begin{cases}
			\text{div}(\gamma|\nabla u|^{p-2}\nabla u)=0 & \text{ in }\Omega,\\
			u=f  & \text{ on }\partial\Omega.
		\end{cases}
	\end{equation}
	Now the inverse problem is to ask whether one can determine the coefficient $\gamma$ uniquely from the knowledge of the (nonlinear) DN map $\Lambda_{\gamma}^p$? Note that in the special case $p=2$ this reduces to the classical Calder\'on problem (see \cite{calderon2006inverse,KohnVogelius,SU87-CalderonProblem-annals}).
	
	Moreover, if $\gamma=1$ then the partial differential operator in \eqref{eq: p Calderon problem summary} becomes the $p\,$-Laplacian $\Delta_pu=\DDiv(|\nabla u|^{p-2}\nabla u)$, which appears in the study of nonlinear dielectrics \cite{Garroni:Kohn:2003, Talbot:Willis:1994:a, Talbot:Willis:1994:b, Kohn:Levy:1998}, plastic moulding \cite{Aronsson:1996}, nonlinear fluids \cite{Antontsev:Rodrigues:2006, Aronsson:Janfalk:1992, Glowinski:Rappaz:2003, Idiart:2008} and others. 
	In \cite{Salo:Zhong:2012} the authors proved by using $p\,$-harmonic functions (i.e. functions solving \eqref{eq: p Laplace type equations}) introduced by Wolff \cite{Wolff:2007} that the nonlinear DN map $\Lambda_{\gamma}^p$ determines uniquely $\gamma$ on the boundary. Later, Brander showed in \cite{Brander:2014} that the DN map also determines the normal derivative $\partial_{\nu}\gamma$ on $\partial\Omega$. These results can be seen as a zeroth and first order analogue of the boundary determination result of Kohn and Vogelius \cite{KohnVogelius} for the Calder\'on problem. Since not all coefficients of the Taylor series around a boundary point, as in the classical Calder\'on problem, are known from the DN map $\Lambda_{\gamma}^p$, it cannot be used to determine real-analytic coefficients in the interior of $\Omega$. Meanwhile, the authors \cite{brander2018monotonicity,Guo-Kar-Salo, Brander:Kar:Salo:2014, Kar-Wang, Brander-Ilmavirta-Kar} studied inverse problems for (weighted) $p\,$-Laplace equations by utilizing monotonicity methods.

	\subsection{Nonlocal inverse problems}
	
	In recent years, many different Calder\'on type inverse problems for nonlocal operators has been studied. The prototypical example is the inverse problem for the fractional Schr\"odinger operator $(-\Delta)^{s}+q$ with $q\in L^{\infty}(\Omega)$, which was first considered in \cite{GSU20} and initiated many of the later developments. The main ingredients in solving this Calder\'on problem is an \emph{Alessandrini identity}, the \emph{UCP} (unique continuation property) and the closely related \emph{Runge approximation}. It is worth noticing that the UCP and the approximation are much stronger than in the local case, which is mainly possible, because solutions to $(-\Delta)^s+q$ are much less rigid than the ones to the local Schr\"odinger equation $-\Delta+q$. By using a similar approach, one can solve a variety of inverse problems for nonlocal operators whose corresponding local counterpart is still open. For further details we refer to the works \cite{bhattacharyya2021inverse,CMR20,CMRU20,GLX,CL2019determining,CLL2017simultaneously,CRZ2022global,cekic2020calderon,feizmohammadi2021fractional,harrach2017nonlocal-monotonicity,harrach2020monotonicity,GRSU18,GU2021calder,ghosh2021non,lin2020monotonicity,LL2020inverse,LL2022inverse,LLR2019calder,LLU2022calder,KLW2021calder,RS17,ruland2018exponential,RZ2022FracCondCounter,RZ2022LowReg,RZ2022unboundedFracCald}. We point out that most of these works consider nonlocal inverse problems in which one recovers lower order coefficients instead of principal order like in the classical Calder\'on problem. On the other hand, in the articles \cite{GU2021calder,LLU2022calder,RZ2022LowReg,RZ2022unboundedFracCald,RZ2022FracCondCounter} the authors study nonlocal inverse problems where one is interested in determining leading order coefficients and hence they can be seen as full nonlocal analogues of the classical Calder\'on problem.
	
	Let us mention that in all the previous inverse problems the leading order operator of the underlying nonlocal PDEs is linear. A first step into the direction of considering nonlinear nonlocal leading order operators was taken in the work \cite{KRZ2022Biharm} by the first and the last author. A crucial advantage of the operators $L$ studied in this work is that they have the UCP, that is, if $u\colon\R^n\to\R$ is a sufficiently regular function and $Lu=u=0$ in an open set $V\subset\R^n$, then $u\equiv 0$ in $\R^n$. But in contrast the UCP for the operators in \eqref{eq: p Laplace type equations} is only known to hold in $n=2$ dimensions but for $n\geq 3$ it is a difficult open problem. Similarly, it is not known whether the operators in \eqref{eq: fractional p laplace type equation} have the UCP.
	
	\subsection{Main results}
	
	The main theorem of this article is the following exterior reconstruction result on the diagonal extending \cite[Proposition~1.5]{CRZ2022global}.
	
	\begin{theorem}[Exterior reconstruction on the diagonal]
		\label{main theorem}
		Let $\Omega \subset \mathbb{R}^n$ be a bounded open set,  $W\subset \Omega_e$ a nonempty open set, $0<s<1$, $1< p<\infty$ and $x_0\in W\Subset \Omega_e$. Then there exists a sequence $(\Phi_N)_{N\in\N}\subset C_c^{\infty}(W)$ such that
		\begin{enumerate}[(i)]
			\item for all $N \in \mathbb{N}$ it holds $[\Phi_N]_{W^{s,p}\left(\mathbb{R}^n\right)}=1$,
			\item
			for all $0\leq t<s$ there holds $\|\Phi_N\|_{W^{t,p}(\R^n)}\to 0$ as $N\to\infty$
			\item
			and $\supp\left(\Phi_N\right) \rightarrow\left\{x_0\right\}$ as $N \rightarrow \infty$.
		\end{enumerate}
		Moreover, if $\sigma\colon \R^n\times \R^n\to \R$ satisfies uniformly elliptic condition  \eqref{eq: uniform ellipticity cond} such that $(\sigma(x,\cdot)\colon W\to \R)_{x\in\R^n}$ is equicontinuous at $x_0$ and $\sigma(\cdot,x_0)\in C(W)$, then there holds 
		\begin{equation}
			\label{eq: convergence on diagonal}
			\sigma(x_0,x_0)=\lim_{N\to\infty}\langle \Lambda_{\sigma}\Phi_N,\Phi_N\rangle.
		\end{equation}
	\end{theorem}
	
	As an immediate consequence of the formula \eqref{eq: convergence on diagonal}, we obtain the following results on exterior determination, exterior stability and global uniqueness for real-analytic coefficients.
	
	\begin{proposition}[Exterior determination on the diagonal]
		\label{prop: exterior determination}
		Let $\Omega \subset \mathbb{R}^n$ be a bounded open set,  $W\subset \Omega_e$ be a nonempty open set, $0<s<1$ and $1< p<\infty$. Assume that $\sigma_j\colon \R^n\times \R^n\to \R$ satisfies the conditions of Theorem~\ref{main theorem} for $j=1,2$ and set $\Sigma_j(x):=\sigma_j(x,x)$ for $x\in \R^n$, $j=1,2$. Suppose $\left.\Lambda_{\sigma_1}f\right|_W=\left.\Lambda_{\sigma_2}f\right|_W$, for all $f\in C_c^{\infty}(W)$, then there holds $\Sigma_1(x)=\Sigma_2(x)$ for all $x\in W$.
	\end{proposition}
	
	\begin{proposition}[Exterior stability on the diagonal]
		\label{prop: exterior stability}
		Let $\Omega \subset \mathbb{R}^n$ be a bounded open set,  $W\subset \Omega_e$ a nonempty open set, $0<s<1$ and $1< p<\infty$. Assume that $\sigma_j\colon \R^n\times \R^n\to \R$ satisfies the conditions of Theorem~\ref{main theorem} for $j=1,2$ and set $\Sigma_j(x)=\sigma_j(x,x)$ for $x\in \R^n$, $j=1,2$. Then we have 
		\[
		\|\Sigma_1-\Sigma_2\|_{L^{\infty}(W)}\leq \|\Lambda_{\sigma_1}-\Lambda_{\sigma_2}\|_{\widetilde{W}^{s,p}(W)\to (\widetilde{W}^{s,p}(W))^*}.
		\]
	\end{proposition}
	
	\begin{proposition}[Global uniqueness on the diagonal]
		\label{prop: uniqueness}
		Let $\Omega \subset \mathbb{R}^n$ be a bounded open set,  $W\subset \Omega_e$ a nonempty open set, $0<s<1$ and $1< p<\infty$. Assume that $\sigma_j\colon \R^n\times \R^n\to \R$ satisfies the conditions of Theorem~\ref{main theorem} for $j=1,2$ and set $\Sigma_j(x)=\sigma_j(x,x)$ for $x\in \R^n$, $j=1,2$. If $\left.\Lambda_{\sigma_1}f\right|_W=\left.\Lambda_{\sigma_2}f\right|_W$ for all $f\in C_c^{\infty}(W)$, and $\Sigma_j$ are real-analytic for $j=1,2$ then $\Sigma_1=\Sigma_2$ in $\R^n$. 
	\end{proposition}
	
	Observe that Proposition~\ref{prop: uniqueness} implies several global uniqueness results when one assumes that the coefficients have a product structure. For example, if for $j=1,2$ the coefficients $\sigma_j(x,y)$ can be written as $\sigma_j(x,y)=F(\gamma_j(x))F(\gamma_j(y))$ for some real analytic functions $\gamma_j\colon\R^n\to\R$, $F\colon\R_+\to\R_+$ satisfying
	\begin{enumerate}[(i)] 
		\item $\gamma_j$ is uniformly elliptic for $j=1,2$,
		\item $F$ is injective
		\item and for any compact interval $[a,b]\subset \R_+$ there exists $c>0$ such that $F(\xi)\geq c$ for all $\xi\in [a,b]$.
	\end{enumerate}
	Then $\left.\Lambda_{\sigma_1}f\right|_W=\left.\Lambda_{\sigma_2}f\right|_W$ for all $f\in C_c^{\infty}(W)$ implies $\gamma_1=\gamma_2$ in $\R^n$. As a special case one could take $F(t)=\sqrt{t}$ and recovers the global uniqueness result in \cite[Theorem~1.3]{CRZ2022global} for real-analytic conductivities.

	\subsection{Organization of the article}
	
	We first recall preliminaries related to the function spaces and nonlocal operators used throughout this work in Section~\ref{sec: Preliminaries}. Afterwards in Section~\ref{sec: The forward problem and DN map} we establish well-posedness of the exterior value problem for \eqref{eq: fractional p laplace type equation} and introduce the related DN map. The proof of the main result, Theorem~\ref{main theorem}, are divided into several steps for better readability and given in Section~\ref{sec: exterior reconstruction}. The proofs of Proposition~\ref{prop: exterior determination}, \ref{prop: exterior stability} and \ref{prop: uniqueness} are given in Section~\ref{sec: consequences of main thm}. 
	
	\section{Preliminaries}
	\label{sec: Preliminaries}
	
	Throughout this article $\Omega \subset \R^n$ is always a bounded open set, where $n\geq 1$ is a fixed positive integer, and $0<s<1$. In this section, we recall the fundamental properties of the classical fractional Sobolev spaces $W^{s,p}(\R^n)$ and their local analogues as well as introduce the nonlocal operators which will be used later on. 
	
	\subsection{Function spaces}
	\label{subsec: Function spaces}
	By $L^0(\Omega)$ we label the space of (Lebesgue) measurable functions on $\Omega$. The classical Sobolev spaces of order $k\in\N$ and integrability exponent $1\leq p\leq \infty$ are denoted by $W^{k,p}(\Omega)$ and for $k=0$ we use the convention $W^{0,p}(\Omega)=L^p(\Omega)$. Moreover, we let $W^{s,p}(\Omega)$ stand for the fractional Sobolev spaces, when $s\in(0,1)$ and $1\leq p < \infty$. These spaces are also called Slobodeckij spaces or Gagliardo spaces. If $1\leq p<\infty$ and $s\in (0,1)$, then they are defined by
	\[
	W^{s,p}(\Omega)\vcentcolon =\left\{\,u\in L^p(\Omega)\,;\, [u]_{W^{s,p}(\Omega)}<\infty\, \right\},
	\]
	where 
	\[
	[u]_{W^{s,p}(\Omega)}\vcentcolon =\left(\int_{\Omega}\int_{\Omega}\frac{|u(x)-u(y)|^p}{|x-y|^{n+s p}}\,dxdy\right)^{1/p}
	\]
	is the so-called Gagliardo seminorm. The fractional Sobolev spaces are naturally endowed with the norm
	\[
	\|u\|_{W^{s,p}(\Omega)}\vcentcolon =\left(\|u\|_{L^{p}(\Omega)}^p+[u]_{W^{s,p}(\Omega)}^p\right)^{1/p}.
	\]
	The space of test functions we are going to use later in the definition of weak solutions to our PDEs is:
	\[
	\widetilde{W}^{s,p}(\Omega)\vcentcolon =\text{closure of }C_c^{\infty}(\Omega) \text{ with respect to } \|\cdot\|_{W^{s,p}(\R^n)}.
	\]
	Similarly, as the classical Sobolev spaces, the spaces $W^{s,p}(\R^n)$ are separable for $1\leq p<\infty$ and reflexive for $1<p<\infty$ (see~\cite[Section~7]{SobolevSpacesCompact}). Since $\widetilde{W}^{s,p}(\Omega)$ is a closed subspace of $W^{s,p}(\R^n)$ it has the same properties. We remark that it is known that $\widetilde{W}^{s,p}(\Omega)$ coincides with the set of all functions $u\in W^{s,p}(\R^n)$ such that $u=0$ almost everywhere (a.e.) in $\Omega^c$, when $\partial\Omega\in C^0$, and with 
	\[
	W^{s,p}_0(\Omega)\vcentcolon =\text{closure of }C_c^{\infty}(\Omega) \text{ with respect to }\|\cdot\|_{W^{s,p}(\Omega)},
	\]
	whenever $\Omega\Subset\R^n$ has a Lipschitz boundary (see~\cite[Section~2]{WienerCriterion}).\\
	
	On these spaces $\widetilde{W}^{s,p}(\Omega)$, the following Poincar\'e inequality holds:
	\begin{proposition}[{Poincar\'e inequality, \cite[Theorem~2.8]{WienerCriterion}}]
		\label{thm: Poincare}
		Let $\Omega\Subset \R^n$, $0<s<1$ and $1<p<\infty$, then there exists a constant $C=C(n,s,p,\text{diam}(\Omega))>0$ such that
		\begin{equation}
			\label{eq: Poincare}
			\|u\|^p_{L^p(\R^n)}\leq C[u]^p_{W^{s,p}(\R^n)}
		\end{equation}
		for all $u\in\widetilde{W}^{s,p}(\Omega)$.
	\end{proposition}
	
	\begin{remark}
		In the above theorem and from now on, we write $V\Subset W$ for two open subset $V, W\subset\R^n$ if $V$ is compactly contained in $W$. By the proof of \cite[Theorem~2.8]{WienerCriterion} it follows that the optimal constant $C_*>0$ in \eqref{eq: Poincare} satisfies $C_*\leq C_1(\mathrm{diam}(\Omega))^{sp}$ for some $C_1=C_1(n,s,p)>0$. Moreover, we used here that by \cite[Theorem~2.8]{WienerCriterion} the estimate \eqref{eq: Poincare} holds for all functions $u\in C_c^{\infty}(\Omega)$, but then the definition of the spaces $\widetilde{W}^{s,p}(\Omega)$ implies that by approximation it holds for all functions in this space. 
	\end{remark}
	
	
	\subsection{Nonlocal opeators}
	\label{subsec: Nonlocal operators}
	
	Next we introduce the fractional $s$-gradient $d_s$, the fractional $s$-divergence $\Div_s$, the fractional $p\,$-Laplacian $(-\Delta)^s_p$ and the weighted fractional $p\,$-Laplacians, which are the main object of study in this article. 
	
	For this purpose, let us denote by $L^0(\bigwedge_{od}^1\R^n)$ the space of measurable off diagonal vector fields, that is, the set of all functions $F\colon\R^n\times \R^n \to \R$ which are measurable with respect to the measure $d\mu\vcentcolon =\frac{dxdy}{|x-y|^n}$ on $\R^n\times \R^n$.
	We next give rigorous definitions of $s$-gradient and $s$-divergence.
	
	\begin{definition}[$s$-gradient]
		For any $0<s<1$, the $s$-gradient $d_s$ is defined by $d_s\colon L^0(\R^n)\to L^0(\bigwedge_{od}^1\R^n)$ such that 
		\begin{equation}
			\label{eq: fractiona gradient}
			d_su(x,y)\vcentcolon = \frac{u(x)-u(y)}{|x-y|^s}.
		\end{equation}
	\end{definition}
	
	Moreover, one can immediately observe that it satisfies the product rule
	\begin{equation}
		\label{eq: product rule}
		d_s(\varphi\psi)(x,y)=\varphi(x)d_s\psi(x,y)+\psi(y)d_s\varphi(x,y),
	\end{equation}
	for a.e. $x,y\in\R^n$ and $\varphi,\psi\colon \R^n\to\R$. Moreover, we call the dual operation to the $s$-gradient $d_s$ the $s$-divergence $\Div_s$. 
	
	\begin{definition}[$s$-divergence]
		For any $0<s<1$, the $s$-divergence is the unbounded operator $\Div_s :\oddvf{\R^n} \to  L^0(\R^n)$ given by
		\begin{equation}
			\label{eq: fractional divergence}
			\left\langle \Div_s F,\varphi \right\rangle =\int_{\R^{2n}}\frac{F(x,y)d_s\varphi(x,y)}{|x-y|^n}\, dxdy, \text{ for all }\varphi\in C_c^{\infty}(\R^n).
		\end{equation}
	\end{definition}
	With these definitions at our disposal, we have a canonical relation to the fractional Laplacian. In fact, there holds $\Div_s\circ \:d_s=(-\Delta)^s$ in the sense that 
	\begin{align*}
		\int_{\R^{2n}}\frac{d_s \varphi(x,y) \, d_s \psi(x,y)}{|x-y|^n} \, dxdy = \int_{\R^n} (-\Delta)^s \varphi(x) \, \psi(x) \, dx,
	\end{align*}
	for all sufficiently regular functions $\varphi,\psi\colon\R^n\to\R$, where $(-\Delta)^s$ denotes the fractional Laplacian of order $s\in (0,1)$ (up to a normalization constant).
	On the other hand, by the above definitions the operator $\Div_s(|d_su|^{p-2}d_su)$ is weakly given by
	\begin{equation}\label{eq: new fractional divergence equation}
		\begin{split}
			&\langle \Div_s(|d_su|^{p-2}d_su),\varphi\rangle\\
			:=&\int_{\R^{2n}}|u(x)-u(y)|^{p-2}\frac{(u(x)-u(y))(\varphi(x)-\varphi(y))}{|x-y|^{n+sp}}\,dxdy
		\end{split}
	\end{equation}
	for all $\varphi\in C_c^{\infty}(\R^n)$. 
	
	Hence, up to normalization, this is precisely the weak formulation of the fractional $p\,$-Laplacian. Furthermore, if the function $u$ is sufficiently smooth then the fractional $p\,$-Laplacian can be calculated in a pointwise sense by
	\begin{equation}
		\label{eq: fractional p Laplacian}
		(-\Delta)^s_pu(x) =C\,\mathrm{P.V.}\,\int_{\R^n}|u(x)-u(y)|^{p-2}\frac{u(x)-u(y)}{|x-y|^{n+sp}}\,dy
	\end{equation}
	for some constant $C=C(n,s,p)>0$ (see~\cite{del_Teso_2021}), where $\mathrm{P.V.}$ denotes the Cauchy principal value. For example one can take $u\in C_c^{\infty}(\R^n)$ with the additional condition $\nabla u\neq 0$ when $p\in (1,\frac{2}{2-s})$. The choice of the constant only becomes important when one wants to prove the following limiting behaviour 
	\begin{equation}\label{eq: limit behaviour}
		\begin{cases}
			(-\Delta)^s_pu\to (-\Delta)^s u &\text{ in }\R^n \text{ as }p\downarrow 2,\\
			(-\Delta)^s_p u\to (-\Delta)_pu &\text{ in }\R^n\text{ as }s\uparrow 1
		\end{cases}
	\end{equation}
	(see \cite[Sectiopn~5]{del_Teso_2021}). 
	Additionally, in the recent article \cite{StabilityFracpLap} (see also \cite{DINEPV-hitchhiker-sobolev} or \cite{BrezisAnotherLookSobolev}) the authors showed that if $\Omega\Subset\R^n$ is a bounded Lipschitz domain then there holds
	\begin{equation}
		\label{eq: convergence of energy functionals}
		\lim_{s\uparrow 1}(1-s)[u]_{W^{s,p}(\R^n)}^p=C(n,p)\|\nabla u\|_{L^p(\Omega)}^p,
	\end{equation}
	for some positive constant $C(n,p)$ and all $u\in W^{1,p}_0(\Omega)$. This observation is then used to show that for fixed $m\in \N$ and $1<p<\infty$ the Dirichlet eigenvalues $(\lambda^s_{m,p}(\Omega))_{s\in (0,1)}$ of the fractional $p\,$-Laplacian (after a suitable renormalization) converge to a multiple of the Dirichlet eigenvalue $\lambda_{m,p}^1(\Omega)$ of the $p\,$-Laplacian as $s\uparrow 1$ and any sequence of normalized eigenfunctions $(u^s_{m,p})_{s\in (0,1)}$ (up to a subsequence) converge to a normalized eigenfunction $u_{m,p}^1$ of the $p\,$-Laplacian as $s\uparrow 1$. 
	
	In this work, we do not pursue the limit behaviour for either $s\uparrow 1$ or $p\to 2$, but we will focus on the exterior determination results for the fractional $p\,$-Laplace equation \eqref{eq: fractional-p-Lap equation}.

	Furthermore, let us point out that there is also a Caffarelli--Silvestre extension type result for the fractional $p\,$-Laplacian, when one replaces the weight $y^{1-2s}$ by $y^{1-sp}$ (see \cite[Section~3]{del_Teso_2021}). Since in this result the function $u$ is required to be $C^2$ regular, the authors do not see an immediate way to generalize to proof of the UCP for the fractional Laplacian in \cite{GSU20} (see also \cite{KRZ2022Biharm}) to the fractional $p\,$-Laplacian. Finally, note that if the fractional $p\,$-Laplacian would have the UCP, then similar methods as in \cite{KRZ2022Biharm} (see also \cite{Guo-Kar-Salo}) could be invoked to prove global uniqueness of the coefficients.
	
	\subsection*{Conventions} Throughout this article, we denote by $B_r(x_0)$, $Q_r(x_0)$ for $r>0$, $x_0\in\R^n$ the open ball of radius $r$ with center $x_0$ and the open cube of side length $2r$ with center $x_0$. Moreover, if $x_0=0$ than we also write $B_r$ and $Q_r$.

	\section{The forward problem and DN map}
	\label{sec: The forward problem and DN map}
	
	In this section we first state an auxilliary lemma which will be of constant use in this work and then establish the well-posedness of the Dirichelt problem related to the fractional $p\,$-Laplace equations \eqref{eq: fractional p laplace type equation}. Finally, we introduce the DN map and show that it induces a continuous map from the trace space to its dual.
	
	\subsection{Auxiliary lemma}
	\label{subsec: Auxiliary lemmas}
	
	\begin{lemma}[{cf.~\cite[eq.~(2.2)]{Simon}, \cite[Lemma 5.1-5.2]{Estimate_p_Laplacian} and \cite[Appendix~A]{Salo:Zhong:2012}}]
		\label{lemma: Auxiliary lemma}
		Let $n\in\N$, $1<p<\infty$, then there exists a constant $c_p>0$ such that for all $x,y\in\R^{n}$, there holds
		\begin{equation}
			\label{eq: superquad auxiliary}
			\left( |x|^{p-2}x-|y|^{p-2}y\right) \cdot (x-y)\geq c_p|x-y|^p
		\end{equation}
		if $p\geq 2$ and
		\begin{equation}
			\label{eq: subquad auxiliary}
			\left(|x|^{p-2}x-|y|^{p-2}y\right)\cdot (x-y)\geq c_p\frac{|x-y|^2}{(|x|+|y|)^{2-p}}
		\end{equation}
		if $1<p<2$. Moreover, for all $1<p<\infty$ we have
		\begin{equation}
			\label{eq: estimate mikko}	
			\begin{split}	
				\left||\xi|^{p-2}\xi-|\eta|^{p-2}\eta\right|\leq C_p &(|\xi|+|\eta|)^{p-2}|\xi-\eta|
			\end{split}
		\end{equation}
		for all $\xi,\eta\in\R^n$ and some constant $C_p>0$.
	\end{lemma}
	
	\subsection{Well-posedness}
	\label{subsec: Well-posedness}
	Next we prove well-posedness of the forward problem. In this work, we use the following notion of weak solutions:
	\begin{definition}[Weak solutions]
		\label{def: weak solutions}
		Let $1<p<\infty$, $0<s<1$, $\Omega\subset\R^n$ be a bounded open set, and assume that $\sigma\colon \R^n\times \R^n\to \R$ satisfies the uniform ellipticity condition \eqref{eq: uniform ellipticity cond}.
		For any $f\in W^{s,p}(\R^n)$, we say that $u\in W^{s,p}(\R^n)$ is a weak solution to
		\begin{equation}
			\label{eq: PDE}
			\begin{cases}
				\Div_s(\sigma|d_su|^{p-2}d_su)=0&\text{ in } \Omega,\\
				u=f & \text{ in } \Omega_e,
			\end{cases}
		\end{equation}
		if $u$ is a distributional solution and $u-f\in \widetilde{W}^{s,p}(\Omega)$.
	\end{definition}
	\begin{remark}
		Observe that if $u$ is a distributional solution of \eqref{eq: PDE}, then there holds
		\[
		\int_{\R^{2n}}\sigma(x,y) |u(x)-u(y)|^{p-2}\frac{(u(x)-u(y))(\varphi(x)-\varphi(y))}{|x-y|^{n+sp}}\,dxdy=0
		\]
		for all $\varphi\in\widetilde{W}^{s,p}(\Omega)$. 
	\end{remark}
	
	\begin{theorem}[Well-posedness]
		\label{thm: well-posedness}
		Let $1<p<\infty$, $0<s<1$, $\Omega\subset\R^n$ be a bounded open set, and assume that $\sigma\colon \R^n\times \R^n\to \R$ satisfies the uniform ellipticity condition \eqref{eq: uniform ellipticity cond}. Then for any $f\in W^{s,p}(\R^n)$, there is a unique solution $u\in W^{s,p}(\R^n)$ of \eqref{eq: PDE}. In fact, it can be characterized as the unique minimizer of the fractional $p\,$-Dirichlet energy
		\begin{equation}
			\label{eq: frac p Dirichlet energy}
			E_{s,p,\sigma}(v)\vcentcolon = \int_{\R^{2n}}\sigma|d^sv|^p\frac{dxdy}{|x-y|^n}
		\end{equation}
		over the class of all $v\in W^{s,p}(\R^n)$ with prescribed exterior data $v=f$ in $\Omega_e$. Moreover, the unique solution satisfies the following estimate
		\begin{equation}
			\label{eq: estimate of solutions}
			[u]_{W^{s,p}(\R^n)}\leq C[f]_{W^{s,p}(\R^n)}
		\end{equation}
		for some $C>0$ depending only on $\lambda$ and $p$.
	\end{theorem}
	
	\begin{remark}
		\label{rem: larger test space}
		Let us observe that we can construct the solution $u$ as the unique minimizer of the functional $E_{s,p,\sigma}$ by the fact that our exterior condition $f$ is such that
		\[
		\int_{\Omega^c\times \Omega^c}\sigma(x,y)\frac{|f(x)-f(y)|^p}{|x-y|^{n+sp}}\,dxdy<\infty.
		\]
		If the exterior condition would be less regular one should only integrate over $\R^{2n}\setminus (\Omega^c\times \Omega^c)$ and impose the exterior value in the sense that $u=f$ a.e. in $\Omega^c$ (cf.~\cite[Section~3]{RosOton16-NonlocEllipticSurvey}).
	\end{remark}
	
	Before giving the proof of this well-posedness result, let us make the following elementary observation. The linear map $T\colon W^{s,p}(\R^n)\to L^p(\R^n)\times L^p(\R^{2n})$ given by
	\[
	u\mapsto \left(u,\frac{|u(x)-u(y)|}{|x-y|^{n/p+s}}\right),
	\]
	where the target space is endowed with the usual product norm, is an isometry. Thus, by arguing as in \cite[Proposition~8.1]{Brezis}, one sees that $W^{s,p}(\R^n)$ is separable in the range $1\leq p<\infty$ and reflexive when $1<p<\infty$. Now since $\widetilde{W}^{s,p}(\Omega)$ is a closed linear subspace of $W^{s,p}(\R^n)$ it follows that it has the same properties on the respective ranges. 
	
	\begin{proof}[Proof of Theorem \ref{thm: well-posedness}]
		We proceed similarly as in \cite[Theorem~5.8]{KRZ2022Biharm}. First we define the convex set
		\[
		\widetilde{W}^{s,p}_f(\Omega)\vcentcolon = \left\{\,u\in W^{s,p}(\R^n)\,:\,u-f\in \widetilde{W}^{s,p}(\Omega)\,\right\}\subset W^{s,p}(\R^n),
		\]
		and observe that it is weakly closed in the reflexive Banach space $W^{s,p}(\R^n)$. To see this assume that $(u_k)_{k\in\N}\subset \widetilde{W}^{s,p}_f(\Omega)$ converges weakly to some $u\in W^{s,p}(\R^n)$ as $k\to \infty$. This implies that the sequence $(u_k-f)_{k\in\N}\subset \widetilde{W}^{s,p}(\Omega)$ converges weakly to $u-f\in W^{s,p}(\R^n)$.

		Next, since weak limits are contained in the weak closure, the weak closure of convex sets coincide with the strong closure and $\widetilde{W}^{s,p}(\Omega)$ is by definition a closed subspace of $W^{s,p}(\R^n)$, we obtain that $u-f\in \widetilde{W}^{s,p}(\Omega)$. Hence, $\widetilde{W}^{s,p}_f(\Omega)$ is weakly closed in $W^{s,p}(\R^n)$. Next note that there holds
		\[
		|a\pm b|^p\geq 2^{1-p}|a|^p-|b|^p
		\]
		for all $a,b\in\R$. Hence, the uniform elliptic condition \eqref{eq: uniform ellipticity cond} of $\sigma$ and the Poincar\'e inequality (Proposition~\ref{thm: Poincare}) yield that
		\[
		\begin{split}
			\left|E_{s,p,\sigma}(u)\right|&\geq \lambda \left(2^{1-p}[u-f]_{W^{s,p}(\R^n)}^p-[f]_{W^{s,p}(\R^n)}^p\right)\\
			&\geq \frac{\lambda}{2^p}\left([u-f]_{W^{s,p}(\R^n)}^p+C\|u-f\|_{L^p(\R^n)}^p\right)-\lambda[f]_{W^{s,p}(\R^n)}^p\\
			&\geq C\|u-f\|^p_{W^{s,p}(\R^n)}-\lambda[f]_{W^{s,p}(\R^n)}^p\\
			&\geq C\|u\|_{W^{s,p}(\R^n)}-c\|f\|_{W^{s,p}(\R^n)}^p
		\end{split}
		\]
		for all $u\in \widetilde{W}^{s,p}_f(\Omega)$ and some constants $C,c>0$ only depending on $\lambda$ and $p$.

		Therefore, the functional $E_{s,p,\sigma}$ is coercive on $\widetilde{W}^{s,p}_f(\Omega)$, in the sense that 
		$$
		|E_{s,p,\sigma}(u)|\to \infty, \text{ when } \|u\|_{W^{s,p}(\R^n)}\to\infty,
		$$ 
		for $u\in \widetilde{W}^{s,p}_f(\Omega)$. We also observe that by the assumptions on $\sigma(x,y)$, the functional $E_{s,p,\sigma}$ is convex and continuous on the closed, convex set $\widetilde{W}^{s,p}_f(\Omega)\subset W^{s,p}(\R^n)$. It is known that this implies that $E_{s,p,\sigma}$ is weakly lower semi-continuous on $\widetilde{W}^{s,p}_f(\Omega)$ (for example, see~\cite[Proposition~2.10]{ConvexityOptimization}). Hence, using \cite[Theorem 1.2]{Variational-Methods} we see that there exists a minimizer $u\in\widetilde{W}_f^{s,p}(\Omega)$ of $E_{s,p,\sigma}$. 
		
		To proceed, let us show that the minimizer $u\in W^{s,p}(\R^n)$ solves \eqref{eq: PDE} in the sense of distributions. Let $\var\in C_c^{\infty}(\Omega)$ then there holds $u_{\eps}\vcentcolon = u+\eps \var\in \widetilde{W}^{s,p}(\Omega)$ for any $\eps \in\R$. Moreover, by H\"older's inequality and the dominated convergence theorem, one can see that $E_{s,p,\sigma}$ is a $C^1-$functional. Hence, the fact that $u$ is a minimizer implies that there holds
		\[
		\begin{split}
			0&=\left.\frac{d}{d\eps}\right|_{\eps=0}E_{s,p,\sigma}(u_{\eps})\\
			&=p\int_{\R^{2n}}\sigma(x,y)|u(x)-u(y)|^{p-2}\frac{(u(x)-u(y))(\var(x)-\var(y))}{|x-y|^{n+sp}}\,dxdy,
		\end{split}
		\]
		and the claim follows. It remains to prove that the minimizer $u$ is unique. Let us first show the assertion for the range $2\leq p<\infty$ and then for $1<p<2$.
		
		\begin{itemize}
			\item[(i)] First suppose that $2\leq p<\infty$. Let $u,v\in W^{s,p}(\R^n)$ and set 
			$$
			\delta_{x,y}w\vcentcolon = w(x)-w(y),
			$$
			for any function $w\colon \R^n\to\R$. Using the estimate \eqref{eq: superquad auxiliary} of Lemma~\ref{lemma: Auxiliary lemma}, we have the following strong monotonicity property
			\begin{equation}
				\label{eq: strong mon prop superquad}
				\begin{split}
					&\int_{\R^{2n}}\sigma\left(|\delta_{x,y}u|^{p-2}\delta_{x,y}u-|\delta_{x,y}v|^{p-2}\delta_{x,y}v\right)\left(\delta_{x,y}u-\delta_{x,y}v\right)\frac{dxdy}{|x-y|^{n+sp}}\\
					\geq &\lambda c_p\int_{\R^{2n}}\frac{|\delta_{x,y}u-\delta_{x,y}v|^p}{|x-y|^{n+sp}}\,dxdy\\
					=&\lambda c_p\int_{\R^{2n}}\frac{|u(x)-u(y)-(v(x)-v(y))|^p}{|x-y|^{n+sp}}\,dxdy\\
					=&\lambda c_p[u-v]_{W^{s,p}(\R^n)}^p.
				\end{split}
			\end{equation}
			Now, if $u,v\in W^{s,p}(\R^n)$ are solutions to  \eqref{eq: PDE} then $u-v\in \widetilde{W}^{s,p}(\Omega)$. Hence the left hand side vanishes and by the Poincar\'e inequality (Theorem~\ref{thm: Poincare}) we can lower bound the right hand side by some positive multiple of $\|u-v\|_{L^p(\R^n)}^p$ but this gives $u=v$ a.e. in $\R^n$.
			
			\item[(ii)] Next let $1<p<2$. Using the same notation as before, we obtain by raising the estimate \eqref{eq: subquad auxiliary} of Lemma~\ref{lemma: Auxiliary lemma} to the power $p/2$ the bound
			\[
			c_p^{p/2}|a-b|^p\leq \left[(|a|^{p-2}a-|b|^{p-2}b)\cdot (a-b)\right]^{p/2}(|a|+|b|)^{(2-p)\frac{p}{2}}
			\]
			for all $a,b\in\R^n$. Now using H\"older's inequality with $\frac{2-p}{2}+\frac{p}{2}=1$ and the uniform ellipticity of $\sigma$, we deduce
			\begin{equation}
				\label{eq: strong mon prop subquad}
				\begin{split}
					&\lambda^{p/2} c^{p/2}_p [u-v]_{W^{s,p}(\R^n)}^p\\
					\leq  & \int_{\R^{2n}}\left|\delta_{x,y}u-\delta_{x,y}v\right|^p\frac{\,dxdy}{|x-y|^{n+sp}}\\
					\leq & \lambda^{p/2} c^{p/2}_p\int_{\R^{2n}}\left[(|\delta_{x,y}u|^{p-2}\delta_{x,y}u-|\delta_{x,y}v|^{p-2}\delta_{x,y}v) (\delta_{x,y}u-\delta_{x,y}v)\right]^{p/2}\\
					&  \cdot \left(|\delta_{x,y}u|+|\delta_{x,y}v|\right)^{(2-p)\frac{p}{2}}\frac{dxdy}{|x-y|^{n+sp}}\\
					\leq &\lambda^{p/2} c^{p/2}_p\|\left(|\delta_{x,y}u|+|\delta_{x,y}v|\right)^{(2-p)\frac{p}{2}}\|_{L^{\frac{2}{2-p}}(\R^n;|x-y|^{-(n+sp)})}\\
					&  \cdot \Big\|[\left(|\delta_{x,y}u|^{p-2}\delta_{x,y}u \right. \\
					& \qquad \left.-|\delta_{x,y}v|^{p-2}\delta_{x,y}v\right) (\delta_{x,y}u-\delta_{x,y}v)]^{p/2}\Big\|_{L^{2/p}(\R^n;|x-y|^{-(n+sp)})} \\
					\leq &c^{p/2}_p2^{p-1} ([u]_{W^{s,p}(\R^n)}^p+[v]_{W^{s,p}(\R^n)}^p)^{\frac{2-p}{2}}\\
					&\cdot  \Bigg(\int_{\R^{2n}}\sigma (|\delta_{x,y}u|^{p-2}\delta_{x,y}u \\ & \qquad -|\delta_{x,y}v|^{p-2}\delta_{x,y}v) (\delta_{x,y}u-\delta_{x,y}v) \frac{dxdy}{|x-y|^{n+sp}}\Bigg)^{p/2}.
				\end{split}
			\end{equation}
			Now, arguing as for the previous case $p\geq 2$, the right hand side vanishes if $u,v\in W^{s,p}(\R^n)$ are solutions to \eqref{eq: PDE} as $u-v\in \widetilde{W}^{s,p}(\Omega)$ and the left hand side can be lower bounded by a positive multiples of $\|u-v\|_{L^p(\R^n)}^p$ by using the Poincar\'e inequality again. Hence, we can conclude that $u=v$ a.e. in $\R^n$.
		\end{itemize}

		To complete the proof, let us establish the estimate \eqref{eq: estimate of solutions}. By Remark~\ref{rem: larger test space}, we can test the equation \eqref{eq: PDE} with any $\var\in \widetilde{W}^{s,p}(\Omega)$ and in particular with $u-f$. Using the uniform ellipticitiy \eqref{eq: uniform ellipticity cond}, writing $u=(u-f)+f$ and applying Young's inequality, we obtain
		\[
		\begin{split}
			\lambda[u]^p_{W^{s,p}(\R^n)}&\leq \int_{\R^{2n}}\sigma(x,y)\frac{|u(x)-u(y)|^p}{|x-y|^{n+sp}}\,dxdy\\
			&\leq \int_{\R^{2n}}\sigma(x,y)|u(x)-u(y)|^{p-2}\frac{(u(x)-u(y))(u(x)-u(y))}{|x-y|^{n+sp}}\,dxdy\\
			&=\int_{\R^{2n}}\sigma(x,y)|u(x)-u(y)|^{p-2}\frac{(u(x)-u(y))(f(x)-f(y))}{|x-y|^{n+sp}}\,dxdy\\
			&\leq \lambda^{-1}\int_{\R^{2n}}\frac{|u(x)-u(y)|^{p-1}\,|f(x)-f(y)|}{|x-y|^{n+sp}}\,dxdy\\
			&\underbrace{\leq \epsilon [u]^p_{W^{s,p}(\R^n)}+C_{\epsilon}\lambda^{-p}[f]_{W^{s,p}(\R^n)}^p}_{ab\leq \epsilon a^p+C_{\epsilon}b^{p'}\text{ with $1/p+1/p'=1$}},
		\end{split}
		\]
		for any $\epsilon>0$, where $C_\epsilon>0$ is a constant depending on $\epsilon$.
		In particular,  by choosing $\epsilon=\lambda/2$, we obtain $[u]_{W^{s,p}(\R^n)}\leq C[f]_{W^{s,p}(\R^n)}$
		for some $C>0$ depending only on $\lambda$ and $p$. This proves the assertion.
	\end{proof}
	
	Next we introduce the abstract trace space:
	\begin{definition}
		\label{eq: abstract trace spaces}
		Let $\Omega\subset\R^n$ be an open set, $0<s<1$ and $1<p<\infty$. Then the abstract trace space $X^{s,p}(\Omega)$ is given by $X^{s,p}(\Omega)\vcentcolon = W^{s,p}(\R^n)/\widetilde{W}^{s,p}(\Omega)$, and we endow it with the quotient norm
		\[
		\|f\|_{X^{s,p}(\Omega)}\vcentcolon = \inf_{u\in \widetilde{W}^{s,p}(\Omega)}\|f-u\|_{W^{s,p}(\R^n)}.
		\]
	\end{definition}
	\begin{remark}
		\label{remark: trace space} Let us point out:
		\begin{itemize}
			\item[(i)] In the above definition and later on we simply write $f$ for an element in $X^{s,p}(\Omega)$ instead of the more precise notation $[f]$. Note that since $\widetilde{W}^{s,p}(\Omega)$ is a closed subspace of the Banach space $W^{s,p}(\R^n)$, the abstract trace space is again a separable, reflexive Banach space in the respective ranges $1\leq p<\infty$ and $1<p<\infty$.
			\item[(ii)] If $\Omega\subset\R^n$ has a bounded Lipschitz continuous boundary then any $f\in W^{s,p}(\Omega_e)$ with $\mathrm{dist}(\supp{f},\partial\Omega)>0$ corresponds to a unique equivalence class in $X^{s,p}(\Omega)$ and the quotient norm is equivalent to the $\|\cdot\|_{W^{s,p}(\Omega_e)}$ norm. In fact, by \cite[Lemma~3.2]{StabilityFracCondEq} the zero extension $\overline{f}$ of $f$ belongs to $W^{s,p}(\R^n)$ and satisfies $\|\overline{f}\|_{W^{s,p}(\R^n)}\leq C\|f\|_{W^{s,p}(\Omega_e)}$. Hence, we implicitly identify below $f$ with the equivalence class $[\overline{f}]$. Moreover, by definition we have $\|\overline{f}\|_{X^{s,p}(\Omega)}\leq \|\overline{f}\|_{W^{s,p}(\R^n)}\leq C\|f\|_{W^{s,p}(\Omega_e)}$. On the other hand, there exists a sequence $u_k\in\widetilde{W}^{s,p}(\Omega)$, $k\in\N$, such that $\|f-u_k\|_{W^{s,p}(\R^n)}\to \|\overline{f}\|_{X^{s,p}(\Omega)}$ as $k\to\infty$. 
			Since $\partial\Omega$ has measure zero, we know that $u_k$ vanish a.e. in $\Omega_e$ and thus there holds
			\[
			\begin{split}
				\|f\|_{W^{s,p}(\Omega_e)}&= \|f-u_k\|_{W^{s,p}(\Omega_e)}\leq \|\overline{f}-u_k\|_{W^{s,p}(\R^n)}\to \|\overline{f}\|_{X^{s,p}(\Omega)}
			\end{split}
			\]
			as $k\to\infty$. This shows that these two norms are equivalent.
		\end{itemize}
		
	\end{remark}
	
	We have the following uniqueness result:
	
	\begin{corollary}
		\label{cor: uniqueness in trace space}
		Let $1<p<\infty$, $0<s<1$, $\Omega\subset\R^n$ a bounded open set and assume that $\sigma\colon \R^n\times \R^n\to \R$ satisfies the uniform ellipticity condition \eqref{eq: uniform ellipticity cond}. Let $u_j\in W^{s,p}(\R^n)$ be the unique solutions of \eqref{eq: PDE} with exterior values $f_j\in W^{s,p}(\R^n)$ for $j=1,2$. If $f_1-f_2\in\widetilde{W}^{s,p}(\Omega)$, then $u_1=u_2$.
	\end{corollary}
	
	\begin{proof}
		By assumption we have $u_1-u_2=(u_1-f_1)-(u_2-f_2)+(f_1-f_2)\in\widetilde{W}^{s,p}(\Omega)$. Hence, arguing as in the proof of Theorem~\ref{thm: well-posedness}, the strong monotonicity properties \eqref{eq: strong mon prop superquad} for $p\geq 2$, and \eqref{eq: strong mon prop subquad} for $1<p<2$, respectively, show that $u_1=u_2$.
	\end{proof}
	
	\subsection{DN maps}
	\label{subsec: DN maps}
	With Theorem \ref{thm: well-posedness} at hand, we can introduce the DN map $\Lambda_{\sigma}$ to formulate the inverse problem.
	If $f\in W^{s,p}(\R^n)$, then the DN map is formally defined by
	\begin{equation}
		\label{eq: pointwise def DN map}
		\Lambda_\sigma(f)=\left.\Div_s(\sigma|d_su_f|^{p-2}d_su_f)\right|_{ \Omega_e},
	\end{equation}
	where $u_f\in W^{s,p}(\R^n)$ is the unique solution of  
	\begin{equation}
		\label{eq: PDE 1}
		\begin{cases}
			\Div_s(\sigma|d_su|^{p-2}d_su)=0& \text{ in }\Omega,\\
			u=f\ &\text{ in } \Omega_e,
		\end{cases}
	\end{equation} 
	(cf.~Theorem~\ref{thm: well-posedness}). As the solution $u_f$ is usually not regular enough to justify the pointwise definition \eqref{eq: pointwise def DN map}, we define it in general in the distributional sense by
	\begin{equation}
		\label{eq: weak formulation DN map}
		\left\langle\Lambda_\sigma(f), g\right\rangle\vcentcolon =\int_{\Omega_e}  \Lambda_\sigma(f) g\,dx\vcentcolon =\int_{\R^{2n}} \sigma\left|d_su\right|^{p-2} d_su \, d_s g \, \frac{dx dy}{|x-y|^n},
	\end{equation}
	where $f,g\in X^{s,p}(\Omega)$. We have:
	
	\begin{proposition}[DN maps]
		\label{prop: DN maps}
		Let $1<p<\infty$, $0<s<1$, $\Omega\subset\R^n$ a bounded open set and assume that $\sigma\colon \R^n\times \R^n\to \R$ satisfies the uniform ellipticity condition \eqref{eq: uniform ellipticity cond}. Then the DN map $\Lambda_{\sigma}$ introduced via \eqref{eq: weak formulation DN map} is a well-defined operator from $X^{s,p}(\Omega)$ to $(X^{s,p}(\Omega))^{\ast}$ and satisfies the estimate
		\begin{equation}
			\label{eq: Estimate DN map}
			\|\Lambda_{\sigma}(f)\|_{(X^{s,p}(\Omega))^*}\leq C\|f\|_{X^{s,p}(\Omega)}^{p-1}
		\end{equation}
		for all $f\in X^{s,p}(\Omega)$ and some $C>0$. Here $(X^{s,p}(\Omega))^{\ast}$ denotes the dual space of $X^{s,p}(\Omega)$.
	\end{proposition}
	\begin{proof}
		First note that by Corollary~\ref{cor: uniqueness in trace space} for any $f\in X^{s,p}(\Omega)$, there is a unique solution $u_f\in W^{s,p}(\R^n)$ of \eqref{eq: PDE 1}. Moreover, changing in the weak formulation of the DN map \eqref{eq: weak formulation DN map}, the function $g\in W^{s,p}(\R^n)$ to $g+\var$ with $\var\in \widetilde{W}^{s,p}(\Omega)$ does not change the value of the DN map as by construction $u_f$ solves \eqref{eq: PDE 1}. Hence, the DN map $\Lambda_{\sigma}$ is well-defined. 
		Finally, by the H\"older's inequality with $\frac{p-1}{p}+\frac{1}{p}=1$, we have
		\[
		\begin{split}
			\left|\left\langle\Lambda_\sigma(f), g\right\rangle\right| &\leq C\left\|d_su\right\|_{L^p(\R^{2n},|x-y|^{-n})}^{p-1}\|d_s g\|_{L^p(\R^{2n},|x-y|^{-n})} \\
			&\leq C\|f\|_{W^{s, p}(\R^n)}^{p-1}\|g\|_{W^{s, p}(\R^n)} 
		\end{split}.
		\]
		for all $f,g\in W^{s,p}(\R^n)$, for some constant $C>0$ independent of $f$ and $g$. Hence, taking the infimum over all representations of $f,g\in X^{s,p}(\Omega)$ and dividing by $\|g\|_{X^{s,p}(\Omega)}$, we obtain
		$$
		\left\|\Lambda_\sigma(f)\right\|_{(X^{s,p}(\Omega))^*} \leq C\|f\|_{X^{s,p}(\Omega)}^{p-1} .
		$$
		This completes the proof.
	\end{proof}
	
	\section{Exterior reconstruction}
	\label{sec: exterior reconstruction}
	
	In this section, we establish the exterior reconstruction result, which is the main theorem of this article.
	
	\begin{lemma}[Exterior conditions]
		\label{lem: exterior conditions} 
		Let $\Omega \subset \mathbb{R}^n$ be a bounded open set, $0<s<1$, $1\leq p<\infty$ and $x_0 \in W \subset \Omega_e$ for an open set $W$. There exists a sequence $\left(\Phi_N\right)_{N \in \mathbb{N}} \subset C_c^{\infty}(W)$ such that
		\begin{enumerate}[(i)]
			\item\label{prop 1 of exterior cond} for all $N \in \mathbb{N}$ it holds $[\Phi_N]_{W^{s,p}\left(\mathbb{R}^n\right)}=1$,
			\item\label{prop 2 of exterior cond}
			for all $0\leq t<s$ there holds $\|\Phi_N\|_{W^{t,p}(\R^n)}\to 0$ as $N\to\infty$
			\item\label{prop 3 of exterior cond}
			and $\supp\left(\Phi_N\right) \rightarrow\left\{x_0\right\}$ as $N \rightarrow \infty$.
		\end{enumerate}
	\end{lemma} 
	\begin{remark}
		Kohn and Vogelius proved a similar result in their celebrated work on boundary determination for the conductivity equation (cf. \cite[Lemma~1]{KohnVogelius} ) for the Sobolev spaces $H^s(\partial \Omega)$, where $\Omega \subset \mathbb{R}^n$ is an open bounded set with smooth boundary (see also \cite[Lemma~5.5]{CRZ2022global}).
	\end{remark}

	\begin{proof}[Proof of Lemma \ref{lem: exterior conditions}]
		By translation and scaling, we may assume that $Q_1\subset W$ and $x_0=0$ without loss of generality. Similarly, as in \cite{KohnVogelius}, we choose any nonzero $\psi\in C_c^{\infty}(\R)$ with $\supp(\psi)\subset (-1,1)$ and let $\Psi$ be the $n$-fold tensor product of $\psi$, that is $\Psi(x)=\prod_{k=1}^n\psi(x_k)$ with $x=(x_1,x_2,\ldots,x_k)$. Next we define the sequence $(\Psi_N)_{N\in\N}$ by $\Psi_N(x)\vcentcolon = \Psi(Nx)$. We clearly have $\Psi_N\in C_c^{\infty}(Q_{1/N})$ and $\supp(\Psi_N)\to \{0\}$. Moreover, by a simple change of variables we have
		\begin{equation}
			\label{eq: Lp norm}
			\|\Psi_N\|_{L^p(\R^n)}=N^{-n/p}\|\Psi\|_{L^p(\R^n)}
		\end{equation}
		and
		\begin{equation}
			\label{eq: Wtp norm}
			[\Psi_N]_{W^{t,p}(\R^n)}=N^{t-n/p}[\Psi]_{W^{t,p}(\R^n)}
		\end{equation}
		for all $0<t<1$, $1\leq p<\infty$ and $N\in\N$. Observe that $[\Psi]_{W^{t,p}(\R^n)}>0$ for all $0\leq t<1$. This is an immediate consequence of $0\neq \psi\in C_c^{\infty}((-1,1))$ and the Poincar\'e inequality (Theorem~\ref{thm: Poincare}). 
		Thus, for all $0\leq t<1$, $1\leq p<\infty$ there exist constants $C_{t,p},C'_{t,p}>0$ such that
		\begin{equation}
			C'_{t,p}N^{t-n/p}\leq \|\Psi_N\|_{W^{t,p}(\R^n)}\leq C_{t,p}N^{t-n/p}
		\end{equation}
		for all $N\in\N$.

		Finally, we introduce for $N\in\N$ the rescaled functions $\Phi_N$ by
		\begin{equation}
			\label{eq: rescaled functions}
			\Phi_N\vcentcolon =\frac{\Psi_N}{[\Psi_N]_{W^{s,p}(\R^n)}}\in C_c^{\infty}(Q_{1/N}).
		\end{equation}
		We clearly have $[\Phi_N]_{W^{s,p}(\R^n)}=1$ and $\supp(\Phi_N)\to \{0\}$. Moreover, from \eqref{eq: Lp norm} and \eqref{eq: Wtp norm} we deduce that
		\[
		\begin{split}
			\|\Phi_N\|_{L^p(\R^n)}&=\frac{\|\Psi_N\|_{L^p(\R^n)}}{[\Psi_N]_{W^{s,p}(\R^n)}}=\frac{N^{-n/p}\|\Psi\|_{L^p(\R^n)}}{N^{s-n/p}[\Psi]_{W^{s,p}(\R^n)}}=N^{-s}\frac{\|\Psi\|_{L^p(\R^n)}}{[\Psi]_{W^{s,p}(\R^n)}}\longrightarrow 0
		\end{split}
		\]
		and
		\[
		\begin{split}
			[\Phi_N]_{W^{t,p}(\R^n)}&=\frac{[\Psi_N]_{W^{t,p}(\R^n)}}{[\Psi_N]_{W^{s,p}(\R^n)}}=\frac{N^{t-n/p}[\Psi]_{W^{t,p}(\R^n)}}{N^{s-n/p}[\Psi]_{W^{s,p}(\R^n)}}=N^{t-s}\frac{[\Psi]_{W^{t,p}(\R^n)}}{[\Psi]_{W^{s,p}(\R^n)}}\longrightarrow 0
		\end{split}
		\]
		as $N\to\infty$, when $0<t<s$. Hence, the sequence $(\Phi_N)_{N\in\N}$ satisfies the properties \ref{prop 1 of exterior cond} -- \ref{prop 3 of exterior cond}.
	\end{proof}
	
	\begin{lemma}
		\label{lem: energy concentration property}
		(Energy concentration property) Let $\Omega \subset \mathbb{R}^n$ be a bounded open set, $0<s<1$, $1< p<\infty$ and $x_0\in W\Subset \Omega_e$. Assume that $(\Phi_N)_{N\in\N}\subset C_c^{\infty}(W)$ is a sequence satisfying the properties \ref{prop 1 of exterior cond} -- \ref{prop 3 of exterior cond} of Lemma~\ref{lem: exterior conditions}. If $\sigma\colon \R^n\times \R^n\to \R$ satisfies the uniform ellipticity condition \eqref{eq: uniform ellipticity cond}, $(\sigma(x,\cdot)\colon W\to \R)_{x\in\R^n}$ is equicontinuous at $x_0$ and $\sigma(\cdot,x_0)\in C(W)$, then we have
		$$
		\sigma\left(x_0,x_0\right)=\lim _{N \rightarrow \infty} E_{s,p,\sigma}\left(\Phi_N\right).
		$$
	\end{lemma}
	\begin{remark}
		We remark that if $\sigma(x,y)=\alpha(x)\beta(y)$ for some uniformly elliptic functions $\alpha,\beta\in L^{\infty}(\R^n)\cap C(W)$ then $\sigma$ satisfies the assumptions of Lemma~\ref{lem: energy concentration property}. 
	\end{remark}

	\begin{proof}[Proof of Lemma \ref{lem: energy concentration property}]
		First note that we can decompose $\sigma(x,y)$ as
		\begin{equation}
			\label{eq: decomposition of kernel}
			\sigma(x,y)=(\sigma(x,y)-\sigma(x,x_0))+(\sigma(x,x_0)-\sigma(x_0,x_0))+\sigma(x_0,x_0), \text{ for all }x,y\in\R^n.
		\end{equation}
		This implies
		\begin{equation}
			\label{eq: decomposition of energy}
			\begin{split}
				E_{s,p,\sigma}(\Phi_N)=&\int_{\R^{2n}}(\sigma(x,y)-\sigma(x,x_0))|d_s\Phi_N|^p\frac{dxdy}{|x-y|^n}\\
				&+\int_{\R^{2n}}(\sigma(x,x_0)-\sigma(x_0,x_0))|d_s\Phi_N|^p\frac{dxdy}{|x-y|^n}\\
				&+\sigma(x_0,x_0)\int_{\R^{2n}}|d_s\Phi_N|^p\frac{dxdy}{|x-y|^n}.
			\end{split}
		\end{equation}
		By \ref{prop 1 of exterior cond} of Lemma~\ref{lem: exterior conditions}, it follows that the last term is equal to $\sigma(x_0,x_0)$. Hence, to establish the assertion it suffice to prove that the two remaining integrals go to zero as $N\to\infty$.

		Next observe that by the product rule for the fractional gradient \eqref{eq: product rule},
		\[\left|d_s\left(\Phi_N \psi\right)\right|^p \leq C\left(\left|\Phi_N(x)\right|^p\left|d_s \psi(x, y)\right|^p+|\psi(y)|^p\left|d_s \Phi_N(x, y)\right|^p\right).
		\]
		If $\psi \in C_b^1\left(\mathbb{R}^n\right)$, then the mean value theorem yields that
		$$
		\begin{aligned}
			&\int_{\mathbb{R}^n} \left|d_s \psi(x, y)\right|^p \frac{d y}{|x-y|^n} \\
			\lesssim &\int_{B_1(x)} \frac{|\psi(x)-\psi(y)|^p}{|x-y|^{n+ sp}} d y+\int_{\mathbb{R}^n \backslash B_1(x)} \frac{|\psi(x)-\psi(y)|^p}{|x-y|^{n+sp}} d y \\
			\lesssim&\|\nabla \psi\|_{L^{\infty}\left(\mathbb{R}^n\right)}^p \int_{B_1(x)} \frac{d y}{|x-y|^{n+p(s-1)}}+\|\psi\|_{L^{\infty}\left(\mathbb{R}^n\right)}^p \int_{\mathbb{R}^n \backslash B_1(x)} \frac{d y}{|x-y|^{n+ sp}} \\
			\lesssim&\left(\int_{B_1(0)} \frac{d z}{|z|^{n+p(s-1)}}+\int_{\mathbb{R}^n \backslash B_1(0)} \frac{d z}{|z|^{n+ps}}\right)\|\psi\|_{C^1\left(\mathbb{R}^n\right)}^p  \\ \lesssim &\|\psi\|_{C^1\left(\mathbb{R}^n\right)}^p,
		\end{aligned}
		$$
		for all $x \in \mathbb{R}^n$. By \ref{prop 2 of exterior cond} of  Lemma~\ref{lem: exterior conditions}, we have $\left\|\Phi_N\right\|_{L^p\left(\mathbb{R}^n\right)} \rightarrow 0$ as $N \rightarrow \infty$, and thus for all $\psi \in C_b^1\left(\mathbb{R}^n\right)$ there holds 
		\begin{equation}
			\label{eq: limit vanishs}
			\int_{\mathbb{R}^{2 n}}\left|\Phi_N(x)\right|^p\left|d_s \psi(x, y)\right|^p \frac{d x d y}{|x-y|^n} \lesssim\|\psi\|_{C^1\left(\mathbb{R}^n\right)}^p\left\|\Phi_N\right\|_{L^p\left(\mathbb{R}^n\right)}^p \rightarrow 0
		\end{equation}
		as $N \rightarrow \infty$.
		Consider now a sequence of functions $\left(\eta_M\right)_{M \in \mathbb{N}} \subset C_c^1\left(\mathbb{R}^n\right)$ such that for all $M \in \mathbb{N}$ it holds $0 \leq \eta_M \leq 1$,
		$\left.\eta_M\right|_{Q_{1 / 2 M}\left(x_0\right)}=1$ and $\left.\eta_M\right|_{\left(Q_{1 / M}\left(x_0\right)\right)^c}=0$. If $N\in\N$ is sufficiently large, then $\eta_M\Phi_N=\Phi_N$ and hence by using \eqref{eq: limit vanishs}, we deduce
		\[
		\begin{split}
			&\limsup_{N\to\infty}\left|\int_{\R^{2n}}(\sigma(x,y)-\sigma(x,x_0))|d_s\Phi_N|^p\frac{dxdy}{|x-y|^n}\right|\\
			=&\limsup_{N\to\infty}\left|\int_{\R^{2n}}(\sigma(x,y)-\sigma(x,x_0))|d_s(\eta_M\Phi_N)|^p\frac{dxdy}{|x-y|^n}\right|\\
			\lesssim& \limsup_{N\to\infty}\left|\int_{\R^{2n}}(\sigma(x,y)-\sigma(x,x_0))|\eta_M(y)|^p|d_s\Phi_N|^p\frac{dxdy}{|x-y|^n}\right|\\
			\lesssim & \sup_{y\in Q_{1/M}(x_0)}\sup_{x\in\R^n}|\sigma(x,y)-\sigma(x,x_0)|
		\end{split}
		\]
		for all $M\in\N$, where we used \ref{prop 1 of exterior cond} of  Lemma~\ref{lem: exterior conditions} in the last step. Using the equicontinuity assumption on $\sigma$, we deduce that 
		\[
		\limsup_{N\to\infty}\left|\int_{\R^{2n}}(\sigma(x,y)-\sigma(x,x_0))|d_s\Phi_N|^p\frac{dxdy}{|x-y|^n}\right|=0.
		\]
		Similarly, we have
		\[
		\begin{split}
			&\limsup_{N\to\infty}\left|\int_{\R^{2n}}(\sigma(x,x_0)-\sigma(x_0,x_0))|d_s\Phi_N|^p\frac{dxdy}{|x-y|^n}\right|\\
			=&\limsup_{N\to\infty}\left|\int_{\R^{2n}}(\sigma(x,x_0)-\sigma(x_0,x_0))|d_s(\eta_M\Phi_N)|^p\frac{dxdy}{|x-y|^n}\right|\\
			\lesssim& \limsup_{N\to\infty}\left|\int_{\R^{2n}}(\sigma(x,x_0)-\sigma(x_0,x_0))|\eta_M(x)|^p|d_s\Phi_N|^p\frac{dxdy}{|x-y|^n}\right|\\
			\lesssim &\sup_{x\in Q_{1/M}(x_0)}|\sigma(x,x_0)-\sigma(x_0,x_0)|
		\end{split}
		\]
		for all $M\in\N$. By the continuity of $x\mapsto \sigma(x,x_0)$ it follows that the last term vanishes as $M\to\infty$. Hence, taking the limit $N\to\infty$ in \eqref{eq: decomposition of energy}, we obtain
		\[
		\sigma(x_0,x_0)=\lim_{N\to\infty}E_{s,p,\sigma}(\Phi_N).
		\]
		This completes the proof.
	\end{proof} 
	
	Now, we observe that if $u_N$ denotes the unique solution of
	\begin{equation}
		\begin{cases}
			\Div_s(\sigma|d_su|^{p-2}d_su)=0& \text{ in }\Omega,\\
			u=\Phi_N\ &\text{ in } \Omega_e,
		\end{cases}
	\end{equation} 
	then by writing $u_N=(u_N-\Phi_N)+\Phi_N\in \widetilde{W}^{s,p}(\Omega)+W^{s,p}(\R^n)$ there holds
	\[
	\int_{\R^{2n}} \sigma\left|d_s u_N\right|^p \frac{dxdy}{|x-y|^n}
	= \int_{\R^{2n}} \sigma \left|d_s u_N\right|^{p-2} d_s u_N \, d_s \Phi_N \frac{dxdy}{|x-y|^n}.
	\]
	Thus, by Proposition~\ref{prop: DN maps} we obtain
	\begin{align*}
		\langle \Lambda_{\sigma}\Phi_N,\Phi_N\rangle 
		= \int_{\R^{2n}} \sigma\left|d_s u_N\right|^p \frac{dxdy}{|x-y|^n} 
		= \int_{\R^{2n}} \sigma\left|d_s u_N\right|^{p-2}d_su_N \,d_s\Phi_N \frac{dxdy}{|x-y|^n},
	\end{align*}
	so that 
	\begin{equation}\label{DN_map_representation}
		\begin{aligned}
			&\langle \Lambda_{\sigma}\Phi_N,\Phi_N\rangle \\
			=& \int_{\R^{2n}} \sigma\left|d_s \Phi_N\right|^p \frac{dxdy}{|x-y|^n} \\
			&+\int_{\R^{2n}} \sigma\left(\left|d_s u_N\right|^{p-2} d_s u_N - |d_s \Phi_N|^{p-2} d_s \Phi_N\right) \, d_s \Phi_N \frac{dxdy}{|x-y|^n} \\
			=& E_{s,p,\sigma}(\Phi_N)\\
			& +\int_{\R^{2n}} \sigma\left(\left|d_s u_N\right|^{p-2} d_s u_N - |d_s \Phi_N|^{p-2} d_s \Phi_N\right) \, d_s \Phi_N \frac{dxdy}{|x-y|^n}. 
		\end{aligned}
	\end{equation}
	We will recover the value of $\sigma$ at $(x_0,x_0)$ by showing that the second term goes to zero as $N \rightarrow \infty $. For this purpose, we show next:
	
	\begin{lemma}\label{Correction_term_estimate}
		Let $\Omega \subset \mathbb{R}^n$ be a bounded open set, $0<s<1$, $1< p<\infty$ and $x_0\in W\Subset \Omega_e$. Assume that $(\Phi_N)_{N\in\N}\subset C_c^{\infty}(W)$ is a sequence satisfying the properties \ref{prop 1 of exterior cond} -- \ref{prop 3 of exterior cond} of Lemma~\ref{lem: exterior conditions}. Let $\sigma\colon \R^n\times \R^n\to \R$ satisfy the uniform elliptic condition \eqref{eq: uniform ellipticity cond}. Then we have
		\begin{equation}
			\label{eq: limit exterior cond and solutions}
			\| u_N -  \Phi_N\|_{W^{s,p}(\mathbb{R}^n)} \rightarrow 0 
		\end{equation}
		when $N\to\infty$, where $u_N\in W^{s,p}(\R^n)$ is the unique solution to 
		\begin{equation}
			\label{eq: solutions corresponding to exterior values}
			\begin{cases}
				\Div_s(\sigma|d_su|^{p-2}d_su)=0 &\text{ in } \Omega,\\
				u=\Phi_N & \text{ in }\Omega_e.
			\end{cases}
		\end{equation} 
	\end{lemma}
	
	\begin{proof}
		We divide the proof into the following two cases:
		
		\begin{itemize}
			\item[(i)] For $2\leq p<\infty$:
			
			By the strong monotonicity property \eqref{eq: strong mon prop superquad}, there holds
			\[
			\begin{split}
				[u_N-\Phi_N]_{W^{s,p}(\R^n)}^p\leq& C\int_{\R^{2n}}\sigma(|\delta_{x,y}u_N|^{p-2}\delta_{x,y}u_N \\
				&\quad -|\delta_{x,y}\Phi_N|^{p-2}\delta_{x,y}\Phi_N)(\delta_{x,y}u_N-\delta_{x,y}\Phi_N)\,\frac{dxdy}{|x-y|^{n+sp}},
			\end{split}
			\]
			for all $N\in\N$ and some $C>0$ only depending on $n,p,\lambda$. Next let us introduce the auxiliary function $\widetilde{\Phi}_N\vcentcolon = u_N-\Phi_N\in \widetilde{W}^{s,p}(\Omega)$. Hence using that $u_N$ solves \eqref{eq: solutions corresponding to exterior values} as well as that $\Phi_N\in C_c^{\infty}(W)$ and $\widetilde{\Phi}_N$ have disjoint supports, we get
			\begin{align*}
				\begin{split}
					&[u_N-\Phi_N]_{W^{s,p}(\R^n)}^p\\
					\leq & C\int_{\R^{2n}}\sigma\left(|\delta_{x,y}u_N|^{p-2}\delta_{x,y}u_N-|\delta_{x,y}\Phi_N|^{p-2}\delta_{x,y}\Phi_N\right)\delta_{x,y}\widetilde{\Phi}_N\,\frac{dxdy}{|x-y|^{n+sp}}\\
					=&\underbrace{-C\int_{\R^{2n}}\sigma|\delta_{x,y}\Phi_N|^{p-2}\delta_{x,y}\Phi_N\delta_{x,y}\widetilde{\Phi}_N\,\frac{dxdy}{|x-y|^{n+sp}}}_{\text{since }u_N \text{ solves \eqref{eq: solutions corresponding to exterior values}}  }\\
					=&-C\int_{\R^{2n}}\sigma|\delta_{x,y}\Phi_N|^{p-2}\left(\Phi_N(x)\widetilde{\Phi}_N(x)+\Phi_N(y)\widetilde{\Phi}_N(y)\right.\\
					&\qquad \qquad \left.-\Phi_N(x)\widetilde{\Phi}_N(y)-\Phi_N(y)\widetilde{\Phi}_N(x)\right)\,\frac{dxdy}{|x-y|^{n+sp}}\\
					=&C\int_{\R^{2n}}\sigma|\delta_{x,y}\Phi_N|^{p-2}\left(\Phi_N(x)\widetilde{\Phi}_N(y)+\Phi_N(y)\widetilde{\Phi}_N(x)\right)\,\frac{dxdy}{|x-y|^{n+sp}}\\
					\vcentcolon= &C(I_1+I_2),
				\end{split}
			\end{align*}
			where 
			\begin{align}
				I_1:=& \int_{\R^{2n}}\sigma|\delta_{x,y}\Phi_N|^{p-2}\Phi_N(x)\widetilde{\Phi}_N(y)\,\frac{dxdy}{|x-y|^{n+sp}},\\
				I_2:=& \int_{\R^{2n}}\sigma|\delta_{x,y}\Phi_N|^{p-2}\Phi_N(y)\widetilde{\Phi}_N(x)\,\frac{dxdy}{|x-y|^{n+sp}}.
			\end{align}

			By H\"older's inequality, the convexity of $x\mapsto |x|^q$ for $q>1$, and Minkowski's inequality, we obtain
			\[
			\begin{split}
				|I_1|&=\left|\int_{\R^{2n}}\sigma|\delta_{x,y}\Phi_N|^{p-2}\Phi_N(x)\widetilde{\Phi}_N(y)\,\frac{dxdy}{|x-y|^{n+sp}}\right|\\
				&\leq C\int_{W}\int_{\Omega}\sigma\left(|\Phi_N(x)|^{p-2}+|\Phi_N(y)|^{p-2}\right)\,|\Phi_N(x)|\,|\widetilde{\Phi}_N(y)|\,\frac{dydx}{|x-y|^{n+sp}}\\
				&=C\|\sigma\|_{L^{\infty}(W\times\Omega)}\int_{W}\int_{\Omega}|\Phi_N(x)|^{p-1}\,|\widetilde{\Phi}_N(y)|\,\frac{dydx}{|x-y|^{n+sp}}\\
				&\leq C\|\sigma\|_{L^{\infty}(W\times\Omega)}\int_W|\Phi_N(x)|^{p-1}\left(\int_{\Omega}\frac{|\widetilde{\Phi}_N(y)|}{|x-y|^{n+sp}}dy\right)dx\\
				&\leq C\|\sigma\|_{L^{\infty}(W\times\Omega)}\||\Phi_N|^{p-1}\|_{L^{\frac{p}{p-1}}(W)}\left\|\int_{\Omega}\frac{|\widetilde{\Phi}_N(y)|}{|x-y|^{n+sp}}dy\right\|_{L^p(W)}\\
				&\leq C\|\sigma\|_{L^{\infty}(W\times\Omega)}\|\Phi_N\|_{L^p(W)}^{p-1}\int_{\Omega}|\widetilde{\Phi}_N(y)|\left(\int_{W}\frac{dx}{|x-y|^{(n+sp)p}}\right)^{1/p}\,dy\\
				&=C\|\sigma\|_{L^{\infty}(W\times\Omega)}\|\Phi_N\|_{L^{p}(W)}^{p-1}\int_{\Omega}|\widetilde{\Phi}_N(y)|\left(\int_{W}\frac{dx}{|x-y|^{(n+sp)p}}\right)^{1/p}\,dy.
			\end{split}
			\]
			Now define $d\vcentcolon = \dist(\Omega,W)>0$. Hence,  H\"older's and Young's inequality imply
			\[
			\begin{split}
				|I_1|&\leq C\|\sigma\|_{L^{\infty}(W\times\Omega)}\|\Phi_N\|_{L^{p}(W)}^{p-1}|\Omega|^{\frac{p-1}{p}}\|\widetilde{\Phi}_N\|_{L^p(\Omega)}\frac{|W|^{1/p}}{d^{n+sp}}\\
				&\leq \eps\|\widetilde{\Phi}_N\|_{L^p(\Omega)}^p+C_{\eps}\|\sigma\|^{\frac{p}{p-1}}_{L^{\infty}(W\times\Omega)}\|\Phi_N\|_{L^{p}(W)}^p|\Omega|\frac{|W|^{\frac{1}{p-1}}}{d^{\frac{p(n+sp)}{p-1}}},
			\end{split}
			\]
			for all $\eps>0$. The same estimate holds for $I_2$ after replacing $\|\sigma\|_{L^{\infty}(W\times \Omega)}$ by $\|\sigma\|_{L^{\infty}(\Omega\times W)}$. Hence, if we choose $\eps>0$ sufficiently small, then by using Poincar\'e's inequality and recalling that $\widetilde{\Phi}_N=u_N-\Phi_N$, the first term on the right hand side can be absorbed on the left hand side to obtain
			\[
			\|u_N-\Phi_N\|^p_{W^{s,p}(\R^n)}\leq C_{\eps}\|\sigma\|^{\frac{p}{p-1}}_{L^{\infty}((W\cup\Omega)\times(W\cup\Omega) )}\|\Phi_N\|_{L^{p}(W)}^p|\Omega|\frac{|W|^{\frac{1}{p-1}}}{d^{\frac{p(n+sp)}{p-1}}}.
			\]
			Now this expression goes to zero as $N$ goes to $\infty$ by Lemma~\ref{lem: exterior conditions}, \ref{prop 2 of exterior cond}.
			
			\item[(ii)] For $1<p<2$:
			
			Applying this time the strong monotonicity property \eqref{eq: strong mon prop subquad} on $u_N-\Phi_N$ gives
			\[
			\begin{split}
				&[u_N-\Phi_N]_{W^{s,p}(\R^n)}^p \\
				\leq & C([u_N]_{W^{s,p}(\R^n)}^p+[\Phi_N]_{W^{s,p}(\R^n)}^p)^{\frac{2-p}{2}}\\
				&\quad\cdot \left(\int_{\R^{2n}}\sigma \left(|\delta_{x,y}u_N|^{p-2}\delta_{x,y}u_N \right.\right. \\
				& \qquad  \quad \left.\left.-|\delta_{x,y}\Phi_N|^{p-2}\delta_{x,y}\Phi_N\right) \left(\delta_{x,y}u_N-\delta_{x,y}\Phi_N\right) \frac{dxdy}{|x-y|^{n+sp}}\right)^{p/2},
			\end{split}
			\]
			for some $C>0$ only depending on $n,p$ and $\lambda$. As in the previous case, this implies
			\[
			\begin{split}
				&[u_N-\Phi_N]_{W^{s,p}(\R^n)}^p\\
				\leq &
				C\|\sigma\|_{L^{\infty}((W\cup\Omega)\times (\Omega\cup W))} \left(\int_{\R^{2n}} |\delta_{x,y}\Phi_N|^{p-2}(|\Phi_N(x)|\,|\widetilde{\Phi}_N(y)| \right. \\
				& \qquad \quad \left. +|\Phi_N(y)|\,|\widetilde{\Phi}_N(x)|) \frac{dxdy}{|x-y|^{n+sp}}\right)^{p/2}  \\
				\leq & C\|\sigma\|_{L^{\infty}((W\cup\Omega)\times (\Omega\cup W))}\left(\int_{\Omega}\int_{W} |\delta_{x,y}\Phi_N|^{p-2}|\Phi_N(x)|\,|\widetilde{\Phi}_N(y)| \frac{dxdy}{|x-y|^{n+sp}}\right)^{p/2}\\
				=&C\|\sigma\|_{L^{\infty}((W\cup\Omega)\times (\Omega\cup W))}\left(\int_{\Omega}\int_{W} |\Phi_N(x)|^{p-1}|\widetilde{\Phi}_N(y)| \frac{dxdy}{|x-y|^{n+sp}}\right)^{p/2},
			\end{split}
			\]
			where we again set $\widetilde{\Phi}_N=u_N-\Phi_N$. Here we used that $u_N$ solves \eqref{eq: solutions corresponding to exterior values}, $\sigma$ is uniformly elliptic, $[\Phi_N]_{W^{s,p}(\R^n)}= 1$ and $\Phi_N$, $\widetilde{\Phi}_N$ have disjoint supports. As in the previous case, this integral can be bounded from above as
			\[
			\begin{split}
				[u_N-\Phi_N]_{W^{s,p}(\R^n)}^p\leq &C\|\sigma\|_{L^{\infty}((W\cup\Omega)\times (\Omega\cup W))}\\
				&\quad \cdot \left(\|\Phi_N\|_{L^{p}(W)}^{p-1}|\Omega|^{\frac{p-1}{p}}\|\widetilde{\Phi}_N\|_{L^p(\Omega)}\frac{|W|^{1/p}}{d^{n+sp}}\right)^{p/2}.
			\end{split}
			\]
			Applying Young's inequality in the from $ab\leq \eps a^2+C_{\eps}b^2$ gives
			\[
			\begin{split}
				[u_N-\Phi_N]_{W^{s,p}(\R^n)}^p\leq &\eps\|\widetilde{\Phi}_N\|_{L^p(\Omega)}^p\\
				& +C_{\eps}\|\sigma\|_{L^{\infty}((W\cup\Omega)\times (\Omega\cup W))}^{2}\left(\|\Phi_N\|_{L^{p}(W)}^{p-1}|\Omega|^{\frac{p-1}{p}}\frac{|W|^{1/p}}{d^{n+sp}}\right)^p
			\end{split}
			\]
			for all $\eps>0$. Using Poincar\'e's inequality and choosing $\eps$ sufficiently small, we can absorb the first term on the left hand side to obtain
			\[
			\|u_N-\Phi_N\|_{W^{s,p}(\R^n)}^p\leq C_{\eps}\|\sigma\|_{L^{\infty}((W\cup\Omega)\times (\Omega\cup W))}^{2}\left(\|\Phi_N\|_{L^{p}(W)}^{p-1}|\Omega|^{\frac{p-1}{p}}\frac{|W|^{1/p}}{d^{n+sp}}\right)^p.
			\]
			By Lemma~\ref{lem: exterior conditions}, \ref{prop 2 of exterior cond} again, we deduce that $\|u_N-\Phi_N\|_{W^{s,p}(\R^N)}\to 0$ as $N\to\infty$.
		\end{itemize}
		This completes the proof.
	\end{proof}
	
	\begin{proposition}\label{Prop_correction_term}
		Suppose that the assumptions of Lemma~\ref{Correction_term_estimate} hold, then we have
		\[
		\lim_{N\to\infty} \int_{\R^{2n}} \sigma\left(\left|d_s u_N\right|^{p-2} d_s u_N - |d_s \Phi_N|^{p-2} d_s \Phi_N\right)  d_s \Phi_N \frac{dxdy}{|x-y|^n} = 0.
		\]
	\end{proposition}
	
	\begin{proof}
		By the estimate \eqref{eq: estimate mikko} of Lemma~\ref{lemma: Auxiliary lemma}, we obtain 
		\[
		\begin{aligned}
			\left|\int_{\R^{2n}} \sigma\left(|d_s u_N|^{p-2} d_s u_N-\left|d_s \Phi_N\right|^{p-2} d_s \Phi_N\right)  d_s \Phi_N \frac{dxdy}{|x-y|^n}\right| \lesssim I,
		\end{aligned}
		\]
		where 
		\[
		I:= \|\sigma\|_{L^{\infty}(\R^{2n})} \int_{\R^{2n}}\left(|d_s u_N|+\left|d_s \Phi_N\right|\right)^{p-2}\left|d_s u_N-d_s \Phi_N\right|\left|d_s \Phi_N\right| \frac{dxdy}{|x-y|^n}.
		\]
		We divide the cases into $p\geq 2$ and $1<p<2$. 
		
		\begin{itemize}
			\item[(i)] For $p \geq 2$: The H\"older's inequality implies that
			\[
			\begin{split}
				I&\lesssim\left(\||d_s u_N|+|d_s \Phi_N|\|^{p-2}_{L^p(\R^{2n},|x-y|^{-n})}\right)\\
				&\qquad\cdot  \left\|d_s (u_N -  \Phi_N)\right\|_{L^p(\R^{2n},|x-y|^{-n})}\left\|d_s \Phi_N\right\|_{L^p(\R^{2n},|x-y|^{-n})}\\
				&\lesssim ([u_N]_{W^{s,p}(\R^n)}+[\Phi_N]_{W^{s,p}(\R^n)})^{p-2}[u_n-\Phi_N]_{W^{s,p}(\R^n)}[\Phi_N]_{W^{s,p}(\R^n)}.
			\end{split}
			\]
			Since, $u_N\in W^{s,p}(\R^n)$ minimizes the fractional $p\,$-Dirichlet energy and $[\Phi_N]_{W^{s,p}(\R^n)}=1$, Lemma~\ref{Correction_term_estimate} shows that this term goes to zero as $N\to\infty$.

			\item[(ii)] For $1<p<2$: We obtain the same estimate from
			\[
			\begin{split}
				&\int_{\R^{2n}}\left(|d_s u_N|+\left|d_s \Phi_N\right|\right)^{p-2}\left|d_s u_N-d_s \Phi_N\right|\left|d_s \Phi_N\right| \frac{dxdy}{|x-y|^n} \\
				\leq &\int_{\R^{2n}}\left|d_s u_N-d_s \Phi_N\right|^{p-1}\left|d_s \Phi_N\right|\frac{dxdy}{|x-y|^n} \\
				\lesssim &\underbrace{\left\|d_s (u_N -  \Phi_N)\right\|_{L^p(\R^n,|x-y|^{-n})}^{p-1}\left\|d_s \Phi_N\right\|_{L^p(\R^n,|x-y|^{-n})}}_{\text{The H\"older's inequality.}}\\
				=&[u_n-\Phi_N]^{p-1}_{W^{s,p}(\R^n)}[\Phi_N]_{W^{s,p}(\R^n)}.
			\end{split}
			\]
			Here we used that $p-2<0$ implies that 
			$$
			\left(|d_s u_N|+\left|d_s \Phi_N\right|\right)^{p-2}\leq |d_s u_N-d_s\Phi_N|^{p-2}.
			$$ Arguing as in the case $p\geq 2$, we see that the last expression goes to zero as $N\to\infty$. 
		\end{itemize}
		Hence, we can conclude the proof.
	\end{proof}

	\begin{lemma}
		\label{lem: DN map as energy}
		Let $\Omega \subset \mathbb{R}^n$ be a bounded open set, $0<s<1$, $1< p<\infty$ and $x_0\in W\Subset \Omega_e$. Assume that $(\Phi_N)_{N\in\N}\subset C_c^{\infty}(W)$ is a sequence satisfying the properties \ref{prop 1 of exterior cond} -- \ref{prop 3 of exterior cond} of Lemma~\ref{lem: exterior conditions} and let $\sigma\colon \R^n\times \R^n\to \R$ satisfy the uniform ellipticity condition \eqref{eq: uniform ellipticity cond}. Then there holds 
		\begin{equation}
			\label{eq: limit of exterior conditions}
			\lim_{N\to\infty}\langle \Lambda_{\sigma}\Phi_N,\Phi_N\rangle =\lim_{N\to\infty}E_{s,p,\sigma}(\Phi_N).
		\end{equation}
	\end{lemma}
	
	\begin{proof}
		As above, denote by $(u_N)_{N\in\N}\subset W^{s,p}(\R^n)$ the unique solutions to \eqref{eq: limit exterior cond and solutions}.
		Now by definition of the DN map $\Lambda_{\sigma}$ and $u_n-\Phi_N\in \widetilde{W}^{s,p}(\Omega)$, there holds
		\[
		\begin{split}
			\langle\Lambda_{\sigma}(\Phi_N),\Phi_N\rangle = E_{s,p,\sigma}(u_N).
		\end{split}
		\]
		Now the result directly follows from \eqref{DN_map_representation} and Proposition \ref{Prop_correction_term}.
	\end{proof}
	
	Finally, we can give the proof of Theorem~\ref{main theorem}.
	
	\begin{proof}[Proof of Theorem~\ref{main theorem}]
		First note that such a sequence $(\Phi_N)_{N\in\N}$ exists by Lemma~\ref{lem: exterior conditions}. By applying Lemma~\ref{lem: energy concentration property} and \ref{lem: DN map as energy}, we deduce
		\[
		\sigma(x_0,x_0)=\lim_{N\to\infty}E_{s,p,\sigma}(\Phi_N)=\lim_{N\to\infty}\langle \Lambda_{\sigma}\Phi_N,\Phi_N\rangle.
		\]
		Hence, we can conclude the proof.
	\end{proof}
	
	\section{Proofs of Proposition~\ref{prop: exterior determination}, \ref{prop: exterior stability} and \ref{prop: uniqueness}}
	\label{sec: consequences of main thm}
	
	The proofs of Proposition~\ref{prop: exterior determination}, \ref{prop: exterior stability} and \ref{prop: uniqueness} can be regarded as corollaries of Theorem \ref{main theorem}, and we give their proofs in the end of this work.
	
	\begin{proof}[Proof of Proposition~\ref{prop: exterior determination}]
		Fix some $x_0\in W$, we can choose a neighborhood $V$ of $x_0$ such that $V\Subset \Omega_e$ and $V\subset W$. Then by Theorem~\ref{main theorem} (with $V$ in place of $W$), we deduce
		\[
		\Sigma_j(x_0)=\sigma_j(x_0,x_0)=\lim_{N\to\infty}\langle \Lambda_{\sigma_j}\Phi_N,\Phi_N\rangle
		\]
		for $j=1,2$. Since, the DN maps of $\sigma_1$ and $\sigma_2$ coincide for smooth functions compactly supported in $W$, we get $\Sigma_1(x_0)=\Sigma_2(x_0)$. As $x_0\in W$ is arbitrary, we get $\Sigma_1=\Sigma_2$ in $W$.
	\end{proof}
	
	\begin{proof}[Proof of Proposition~\ref{prop: exterior stability}]
		Let $x_0\in W$ and choose as above a neighborhood $V\Subset \Omega_e$ such that $x_0\in V\subset W$. Then by Theorem~\ref{main theorem} (with $V$ in place of $W$), we have
		\[
		\begin{split}
			&|\Sigma_1(x_0)-\Sigma_2(x_0)|\\
			=&\lim_{N\to\infty}\left|\langle (\Lambda_{\sigma_1}-\Lambda_{\sigma_2}\Phi_N,\Phi_N\rangle\right|\\
			\leq &\limsup_{N\to\infty}\|\Lambda_{\sigma_1}-\Lambda_{\sigma_2}\|_{\widetilde{W}^{s,p}(V)\to (\widetilde{W}^{s,p}(V))^*}\|\Phi_N\|_{W^{s,p}(\R^n)}\|\Phi_N\|_{W^{s,p}(\R^n)}\\
			\leq &\|\Lambda_{\sigma_1}-\Lambda_{\sigma_2}\|_{\widetilde{W}^{s,p}(W)\to (\widetilde{W}^{s,p}(W))^*}.
		\end{split}
		\]
		In the last step, we used that by Theorem~\ref{main theorem} the sequence $(\Phi_N)_{N\in\N}\subset C_c^{\infty}(V)$ satisfies $\|\Phi_N\|_{L^p(\R^n)}\to 0$ as $N\to\infty$ and $[\Phi_N]_{W^{s,p}(\R^N)}=1$ for all $N\in\N$. Since, $x_0\in W$ was arbitrary and the right hand side is independent of $x_0$, we deduce
		\[
		\|\Sigma_1-\Sigma_2\|_{L^{\infty}(W)}\leq \|\Lambda_{\sigma_1}-\Lambda_{\sigma_2}\|_{\widetilde{W}^{s,p}(W)\to (\widetilde{W}^{s,p}(W))^*}
		\]
		and we can conclude the proof.
	\end{proof}
	
	\begin{proof}[Proof of Proposition~\ref{prop: uniqueness}]
		Proposition~\ref{prop: exterior determination} implies $\Sigma_1=\Sigma_2$ on the nonempty open set $W\subset\Omega_e$. With the real-analyticity of $\Sigma_j$ at hand, $j=1,2$, this immediately implies $\Sigma_1=\Sigma_2$ in $\R^n$.
	\end{proof}

	\bibliography{refs} 
	
	\bibliographystyle{alpha}
	
\end{document}